\documentclass[12pt]{article}
\usepackage{amsthm, amsmath, amssymb, amsfonts, latexsym}
\date{}

\newtheorem{theorem}{\bf \indent Theorem}[section]

\newtheorem{lemma}{\bf \indent Lemma}[section]

\theoremstyle{remark}
\newtheorem{remark}{\bf \indent Remark}[section]

\numberwithin{equation}{section}

\newcommand{\R}{{\mathbb R}}

\newcommand{\D}{{\mathcal D}}
\newcommand{\E}{{\mathcal E}}
\newcommand{\const}{\mathrm{const}}

\title{On self-similar solutions of a multi-phase Stefan problem}

\author{Evgeny Yu. Panov \\ Yaroslav-the-Wise Novgorod State University, \\ Veliky Novgorod, Russian Federation}

\begin{document}

\maketitle

\begin{abstract}
We study self-similar solutions of a multi-phase Stefan problem, first in the case of one space variable, and then in the radial multidimensional case. In both these cases we prove that a nonlinear algebraic system for determination of the free boundaries is gradient one and the corresponding potential is an explicitly written coercive function. Therefore, there exists a minimum point of the potential, coordinates of this point determine free boundaries and provide the desired solution. Moreover, in one-dimensional case the potential is proved to be strictly convex and this implies the uniqueness of the solution. In contrary, in the multidimensional case the potential is not convex but the uniqueness of our solution remains true, it follows from the general theory.
Bibliography$:$
$3$ titles.
\end{abstract}

\section{Stefan problem. One-dimensional case}\label{sec1}

In a half-plane $\Pi=\{ \ (t,x) \ | \ t>0, x\in\R \ \}$ we consider the multi-phase Stefan problem for the heat equation
\begin{equation}\label{1}
u_t=a_i^2u_{xx}, \quad u_i<u<u_{i+1},
\end{equation}
where $u_-=u_0<u_1<\cdots<u_m<u_{m+1}=u_+$, $u_i$, $i=1,\dots,m$ being the temperatures of phase transitions, $a_i>0$,
$i=0,\dots,m$, are the diffusivity constants. On the unknown lines $x=x_i(t)$ of phase transitions where $u=u_i$ the following Stefan condition
\begin{equation}\label{St}
d_ix_i'(t)+k_iu_x(t,x_i(t)+)-k_{i-1}u_x(t,x_i(t)-)=0
\end{equation}
is postulated, where $k_i>0$ is the thermal conductivity of the $i$-th phase, while $d_i\ge 0$ is the Stefan number (the latent specific heat) for the $i$-th phase transition. In (\ref{St}) the unilateral limits $u_x(t,x_i(t)+)$, $u_x(t,x_i(t)-)$ on the line $x=x_i(t)$ are taken from the domain corresponding to the warmer/colder phase, respectively. By the physical reasons,
the Stefan numbers $d_i$ should be positive. We will study even more general case $d_i\ge 0$.
In this case the problem (\ref{1}), (\ref{St}) is well-posed for $u_-< u< u_+$ and reduces to a degenerate nonlinear diffusion  equation,  see \cite[Chapter 5]{LSU}.
We will study the Cauchy problem with the Riemann initial data
\begin{equation}\label{2}
u(0,x)=\left\{\begin{array}{lr} u_-, & x<0, \\ u_+, & x>0. \end{array}\right.
\end{equation}
By the invariance of our problem under the transformation group
$(t,x)\to (\lambda^2 t, \lambda x)$, $\lambda\in\R$, $\lambda\not=0$, it is natural to seek a self-similar solution of problem (\ref{1}), (\ref{St}), (\ref{2}), which has the form $u(t,x)=v(\xi)$, $\xi=x/\sqrt{t}$. For the heat equation
$u_t=a^2 u_{xx}$ a self-similar solution must satisfy the linear ODE $a^2v''=-\xi v'/2$, the general solution of which is
\[
v=C_1F(\xi/a)+C_2, \ C_1,C_2=\const, \mbox{ where } F(\xi)=\frac{1}{2\sqrt{\pi}}\int_{-\infty}^\xi e^{-s^2/4}ds.
\]
This allows to write our solution in the form
\begin{align}\label{3}
v(\xi)=u_i+\frac{u_{i+1}-u_i}{F(\xi_{i+1}/a_i)-F(\xi_i/a_i)}(F(\xi/a_i)-F(\xi_i/a_i)), \\ \nonumber
\xi_i<\xi<\xi_{i+1}, \ i=0,\ldots,m,
\end{align}
where $-\infty=\xi_0<\xi_1<\cdots<\xi_m<\xi_{m+1}=+\infty$ and we agree that $F(-\infty)=0$, $F(+\infty)=1$.
The parabolas $\xi=\xi_i$, $i=1,\ldots,m$, where $u=u_i$, are free boundaries. They must be determined by conditions (\ref{St}).
In the variable $\xi$ these conditions have the form (cf. \cite[Chapter XI]{CJ})
\begin{equation}\label{4}
d_i\xi_i/2+\frac{k_i(u_{i+1}-u_i)F'(\xi_i/a_i)}{a_i(F(\xi_{i+1}/a_i)-F(\xi_i/a_i))}-
\frac{k_{i-1}(u_i-u_{i-1})F'(\xi_i/a_{i-1})}{a_{i-1}(F(\xi_i/a_{i-1})-F(\xi_{i-1}/a_{i-1}))}=0,
\end{equation}
$i=1,\ldots,m$. To investigate this nonlinear system, we notice that it is a gradient one and coincides with the equality
$\nabla E(\bar\xi)=0$, where the function
\begin{align}\label{5}
E(\bar\xi)=-\sum_{i=0}^m  k_i(u_{i+1}-u_i)\ln (F(\xi_{i+1}/a_i)-F(\xi_i/a_i))+\sum_{i=1}^m d_i\xi_i^2/4, \\ \nonumber \bar\xi=(\xi_1,\ldots,\xi_m)\in\Omega,
\end{align}
the open convex domain $\Omega\subset\R^m$ is given by the inequalities $\xi_1<\cdots<\xi_m$.
Observe that $E(\bar\xi)\in C^\infty(\Omega)$. Since the function $F(x)$ takes values in the interval $(0,1)$ all the terms
in expression (\ref{5}) are nonnegative while some of them are strictly positive. Therefore, $E(\bar\xi)>0$.

\subsection{Coercivity of $E$}

Let us introduce the sub-level sets \[\Omega_c=\{ \ \bar\xi\in\Omega \ | \ E(\bar\xi)\le c \ \}, \quad c>0.\]

\begin{theorem}[coercivity]\label{th1}
The sets $\Omega_c$ are compact for each $c>0$.
In particular, the function $E(\bar\xi)$ reaches its minimal value.
\end{theorem}

\begin{proof}
If $\bar\xi=(\xi_1,\ldots,\xi_m)\in\Omega_c$ then
\begin{equation}\label{co1a}
 -k_i(u_{i+1}-u_i)\ln (F(\xi_{i+1}/a_i)-F(\xi_i/a_i))\le E(\bar\xi)\le c, \quad i=0,\ldots,m.
\end{equation}
It follows from (\ref{co1a}) with $i=0$ that $F(\xi_1/a_0)\ge e^{-c/(k_0(u_1-u_0))}$, which implies
the low bound $\xi_1\ge r_1=a_0F^{-1}(e^{-c/(k_0(u_1-u_0))})$.
Similarly, we derive from (\ref{co1a}) with $i=m$ and from the identity $1-F(\xi_m/a_m)=F(-\xi_m/a_m)$ that
$\xi_m\le r_2=-a_mF^{-1}(e^{-c/(k_m(u_{m+1}-u_m))})$.
Since all other coordinates of $\bar\xi$ are situated between $\xi_1$ and $\xi_m$ we conclude that
\[|\bar\xi|_\infty\doteq\max_{i=1,\ldots,m}|\xi_i|\le r=\max(|r_1|,|r_2|)\] and the set $\Omega_c$ is bounded.

Further, it follows from (\ref{co1a}) that for all $i=1,\ldots,m-1$
\begin{equation}\label{12}
F(\xi_{i+1}/a_i)-F(\xi_i/a_i)\ge \delta\doteq\exp(-c/\alpha)>0,
\end{equation}
where $\displaystyle\alpha=\min_{i=1,\ldots,m-1}k_i(u_{i+1}-u_i)>0$. Since $F'(\xi)=\frac{1}{2\sqrt{\pi}}e^{-\xi^2/4}<1$, the function $F(\xi)$ is Lipschitz with constant $1$, and it follows from (\ref{12}) that
\[
(\xi_{i+1}-\xi_i)/a_i\ge F(\xi_{i+1}/a_i)-F(\xi_i/a_i)\ge \delta, \quad i=1,\ldots,m-1,
\]
and we obtain the estimates $\xi_{i+1}-\xi_i\ge\delta_1=\delta\min a_i$. Thus, the set $\Omega_c$ is contained in a compact
\[
K=\{ \ \bar\xi=(\xi_1,\ldots,\xi_m)\in\R^m \ | \ |\bar\xi|_\infty\le r, \ \xi_{i+1}-\xi_i\ge\delta_1 \ \forall i=1 ,\ldots,m-1 \ \}.
\]
Since $E(\bar\xi)$ is continuous on $K$, the set $\Omega_c$ is a closed subset of $K$ and therefore is compact. For $c>N\doteq\inf E(\bar\xi)$, this set is not empty and the function $E(\bar\xi)$ reaches on it a minimal value, which is evidently equal $N$.
\end{proof}

We have established the existence of minimal value $E(\bar\xi_0)=\min E(\bar\xi)$. At the point $\bar\xi_0$ the required condition $\nabla E(\bar\xi_0)=0$ is satisfied, and $\bar\xi_0$ is a solution of system (\ref{4}). The coordinates
of $\bar\xi_0$ determine the solution (\ref{3}) of our Stefan-Riemann problem. Thus, we establish the following existence result.

\begin{theorem}\label{th2}
There exists a self-similar solution (\ref{3}) of problem (\ref{1}), (\ref{St}), (\ref{2}).
\end{theorem}

\subsection{Convexity of the Function $E$ and Uniqueness of the Solution}

In this section we prove that the function $E(\bar\xi)$ is strictly convex. Since a strictly convex function can have at most one critical point (and it is necessarily a global minimum) the system (\ref{4}) has at most one solution, that is, a self-similar solution (\ref{3}) of problem (\ref{1}), (\ref{St}), (\ref{2}) is unique. We will need the following simple lemma.

\begin{lemma}\label{lem1}
The function $P(x,y)=-\ln (F(x)-F(y))$ is strictly convex in the half-plane $x>y$.
\end{lemma}

\begin{proof}
The function $P(x,y)$ is infinitely differentiable in the domain $x>y$. To prove the lemma, we need to establish that the Hessian $D^2 P$ is positive definite at every point. By the direct computation we find
\begin{align*}
\frac{\partial^2}{\partial x^2} P(x,y)=\frac{(F'(x))^2-F''(x)(F(x)-F(y))}{(F(x)-F(y))^2}, \\
\frac{\partial^2}{\partial y^2} P(x,y)=\frac{(F'(y))^2-F''(y)(F(y)-F(x))}{(F(x)-F(y))^2}, \
\frac{\partial^2}{\partial x\partial y} P(x,y)=-\frac{F'(x)F'(y)}{(F(x)-F(y))^2}.
\end{align*}
We have to prove positive definiteness of the matrix $Q=(F(x)-F(y))^2 D^2 P(x,y)$ with the components
\begin{align*}
Q_{11}=(F'(x))^2-F''(x)(F(x)-F(y)), \\ Q_{22}=(F'(y))^2-F''(y)(F(y)-F(x)), \ Q_{12}=Q_{21}=-F'(x)F'(y).
\end{align*}
Since $F'(x)=e^{-x^2/4}$, then $F''(x)=-\frac{x}{2}F'(x)$ and the diagonal elements of this matrix can be written in the form
\begin{align*}
Q_{11}=F'(x)(\frac{x}{2}(F(x)-F(y))+F'(x))= \\ F'(x)(\frac{x}{2}(F(x)-F(y))+(F'(x)-F'(y)))+F'(x)F'(y), \\
Q_{22}=F'(y)(\frac{y}{2}(F(y)-F(x))+(F'(y)-F'(x)))+F'(x)F'(y).
\end{align*}
By Cauchy mean value theorem there exists such a value $z\in (y,x)$ that
\[
\frac{F'(x)-F'(y)}{F(x)-F(y)}=\frac{F''(z)}{F'(z)}=-z/2.
\]
Therefore,
\begin{align*}
Q_{11}=F'(x)(F(x)-F(y))(x-z)/2+F'(x)F'(y), \\ Q_{22}=F'(y)(F(x)-F(y))(z-y)/2+F'(x)F'(y),
\end{align*}
and it follows that $Q=R_1+F'(x)F'(y)R_2$, where $R_1$ is a diagonal matrix with the positive diagonal elements
$F'(x)(F(x)-F(y))(x-z)/2$, $F'(y)(F(x)-F(y))(z-y)/2$ while $R_2=\left(\begin{smallmatrix} 1 & -1 \\ -1 & 1\end{smallmatrix}\right)$. Since $R_1>0$, $R_2\ge 0$, then the matrix $Q>0$, as was to be proved.
\end{proof}

\begin{remark}\label{rem1}
In addition to Lemma~\ref{lem1} we observe that the functions $P(x,-\infty)$, $P(+\infty,x)=P(-x,-\infty)$
of single variable $x$ are strictly convex. In fact, it is sufficient to prove the strict convexity of the function
$P(x,-\infty)=-\ln F(x)$. By Lemma~\ref{lem1} in the limit as $y\to-\infty$ we obtain that this function is convex, moreover,
\[
(F(x))^2\frac{d^2}{dx^2}P(x,-\infty)=F'(x)(\frac{x}{2}F(x)+F'(x))=\lim_{y\to-\infty}Q_{11}\ge 0.
\]
Since $F'(x)>0$, we find, in particular, that $\frac{x}{2}F(x)+F'(x)\ge 0$.
If $\frac{d^2}{dx^2}P(x,-\infty)=0$ at some point $x=x_0$ then $0=\frac{x_0}{2}F(x_0)+F'(x_0)$ is the minimum of the nonnegative function $\frac{x}{2}F(x)+F'(x)$. Therefore, its derivative $(\frac{x}{2}F+F')'(x_0)=0$. Since $F''(x)=-\frac{x}{2}F'(x)$, this derivative
\[
(\frac{x}{2}F+F')'(x_0)=F(x_0)/2+\frac{x_0}{2}F'(x_0)+F''(x_0)=F(x_0)/2>0.
\]
But this contradicts our assumption. We conclude that $\frac{d^2}{dx^2}P(x,-\infty)>0$ and the function $P(x,-\infty)$ is strictly convex.
\end{remark}
Now we are ready to prove the expected convexity of $E(\bar\xi)$.

\begin{theorem}
The function $E(\bar\xi)$ is strictly convex.
\end{theorem}

\begin{proof}
We introduce the functions
\[
E_i(\bar\xi)=-k_i(u_{i+1}-u_i)\ln (F(\xi_{i+1}/a_i)-F(\xi_i/a_i)), \quad i=0,\ldots,m.\]
By Lemma~\ref{lem1} and Remark~\ref{rem1} all these functions are convex.
Since
\[
E(\bar\xi)=\sum_{i=0}^{m-1} E_i(\bar\xi)+ E_m(\bar\xi)+\sum_{i=1}^m d_i\xi_i^2/4
\]
and all the functions in this sum are convex, it is sufficient to prove the strong convexity of the sum
\[
\tilde E(\bar\xi)=\sum_{i=0}^{m-1} E_i(\bar\xi).
\]
By Lemma~\ref{lem1} and Remark~\ref{rem1} all the terms in this sum are convex function. Therefore, the function $\tilde E$ is convex as well. To prove the strict convexity, we assume that for some vector $\zeta=(\zeta_1,\ldots,\zeta_m)\in\R^m$.
\begin{equation}\label{deg}
D^2 \tilde E(\bar\xi)\zeta\cdot\zeta=\sum_{i,j=1}^m \frac{\partial^2 \tilde E(\bar\xi)}{\partial\xi_i\partial\xi_j}\zeta_i\zeta_j=0
\end{equation}
Since
\[
0=D^2 \tilde E(\bar\xi)\zeta\cdot\zeta=\sum_{i=0}^{m-1} D^2 E_i(\bar\xi)\zeta\cdot\zeta
\]
while all the terms are nonnegative, we conclude that
\begin{equation}\label{deg1}
D^2 E_i(\bar\xi)\zeta\cdot\zeta=0, \quad i=0,\ldots,m-1.
\end{equation}
By Lemma~\ref{lem1} for $i=1,\ldots,m-1$ the function $E_i(\bar\xi)$ is strictly convex as a function of two variables $\xi_i,\xi_{i+1}$ and it follows from (\ref{deg1}) that $\zeta_i=\zeta_{i+1}=0$, $i=1,\ldots,m-1$. Observe that in the case $m=1$ there are no such $i$. In this case we apply (\ref{deg1}) for $i=0$. Taking into account Remark~\ref{rem1}, we find that $E_0(\bar\xi)$ is a strictly convex function of the single variable $\xi_1$, and it follows from (\ref{deg1}) that $\zeta_1=0$. In any case we obtain that the vector $\zeta=0$. Thus, relation (\ref{deg}) can hold only for zero $\zeta$, that is, the matrix $D^2\tilde E(\bar\xi)$ is (strictly) positive definite, and the function $\tilde E(\bar\xi)$ is strictly convex. This completes the proof.
\end{proof}
Theorems~\ref{th1},\ref{th2} imply our main result.

\begin{theorem}\label{th3}
There exists a unique self-similar solution (\ref{3}) of problem (\ref{1}), (\ref{St}), (\ref{2}), and it correspond to the minimum of strictly convex and coercive function (\ref{5}).
\end{theorem}

\begin{remark}\label{rem2}
In recent paper \cite{Pan1} the problem (\ref{1}), (\ref{St}), (\ref{2}) was studied in the case of arbitrary (possibly negative) latent specific heats $d_i$. It was found a necessary and sufficient condition for the coercivity of $E(\bar\xi)$, as well as a stronger sufficient condition of its strict convexity.
\end{remark}

\section{Multidimensional Stefan problem}

Now we consider the multidimensional case, $x\in\R^n$, $n>1$. In the half-space $\Pi=\{ \ (t,x) \ | \ t>0, x\in\R^n \ \}$
we study the Stefan problem
\begin{equation}\label{m1}
u_t=a_i^2\Delta u, \quad u_i<u<u_{i+1}, \ i=0,\ldots,m,
\end{equation}
where $u_0<u_1<\cdots<u_m<u_{m+1}=+\infty$.
On a phase transition hyper-surface $S$, where $u=u_i$, $i=1,\ldots,m$, the Stefan condition reads
\begin{equation}\label{mSt}
-d_i\nu_t+(k_i\nabla_x u_+-k_{i-1}\nabla_x u_-)\cdot\nu_x=0,
\end{equation}
where $\nu=(\nu_t,\nu_x)\in\R\times\R^n$, $\nu\not=0$ is a normal vector to $S$ directed from the $(i-1)$-th phase to the $i$-th phase (so that the temperature increases in the direction of $\nu$), $\nabla_x u_-$, $\nabla_x u_+$ are unilateral limits of the gradient $\nabla_x u$ from the colder (respectively, the warmer) side. As in section~\ref{sec1}, the constants $k_i>0$, $i=0,\ldots,m$, are the thermal conductivities, $d_i\ge 0$, $i=1,\ldots,m$, are the latent specific heats.

We are interesting in self-similar radial solutions \[u=v(r/\sqrt{t}), \ r=|x|=(x_1^2+\cdots +x_n^2)^{1/2},\] defined for $r>0$. Putting the function $v(r/\sqrt{t})$ into the heat equation $u_t=a^2\Delta u$, we find that this equation reduces to the linear ODE
\[
a^2v''+\left(a^2\frac{n-1}{\xi}+\frac{\xi}{2}\right)v'=0. \quad\xi=r/\sqrt{t}.
\]
The general solution of this equation is the following:
\begin{equation}\label{form}
v=C_1G(\xi/a)+C_2, \quad C_1,C_2=\const, \quad G(y)=\int_y^{+\infty} s^{1-n}e^{-s^2/4}ds, \ y>0.
\end{equation}
Remark that
\[
G(+\infty)\doteq\lim_{y\to +\infty}G(y)=0, \ \lim_{y\to 0+}G(y)=+\infty.
\]
Moreover, as $y\to 0+$
\begin{equation}\label{lim0}
G(y)\sim \E_n(y)=\left\{\begin{array}{lcr} \frac{1}{(n-2) y^{n-2}} & , & n\ge 3, \\ -\ln y & , & n=2, \end{array}\right.
\end{equation}
so that $-\frac{1}{\omega_n}\E_n(|x|)$ is the fundamental solution of the Laplace operator in $\R^n$, $\omega_n$ being the surface area of a unit sphere in $\R^n$.
We are going to construct a solution to problem (\ref{m1}), (\ref{mSt}) with the constant initial data
\begin{equation}\label{m2}
u(0,x)=u_0,
\end{equation}
and with the following growth condition at $x=0$:
\begin{equation}\label{gr}
u(t,x)\sim A\E_n(|x|/(a_m\sqrt{t})) \ \mbox{ as } x\to 0,
\end{equation}
where $A$ is a positive constant (we will demonstrate below that this condition corresponds to a positive heat source at the point $x=0$). It follows from (\ref{form}), (\ref{lim0}), (\ref{gr}) that in a vicinity of zero our solution has the form
\begin{equation}\label{solm}
v(\xi)=u_m+A(G(\xi/a_m)-G(\xi_m/a_m)), \quad 0<\xi<\xi_m,
\end{equation}
where $\xi_m$ corresponds to the unknown phase transition surface $|x|=\xi_m\sqrt{t}$ with temperature $u_m$.
Notice that the function $u=v(|x|/\sqrt{t})$ satisfies the inhomogeneous heat equation
\begin{equation}\label{sr1}
u_t-a_m^2\Delta u=A\omega_na_m^nt^{\frac{n}{2}-1}\delta(x)
\end{equation}
in the sense of distribution on $\Pi$ (in $\D'(\Pi)$), with the positive point heat source $A\omega_na_m^nt^{\frac{n}{2}-1}\delta(x)$ (in the case $n=2$ it does not depend on time). Here, as usual, $\delta(x)$ denotes the Dirac $\delta$-function. In fact, applying the distribution $u_t-a_m^2\Delta u$ to a test function $f=f(t,x)\in C_0^\infty(\Pi)$ and integrating by parts in the domain $|x|>\delta>0$, we obtain that
\begin{align}\label{sr2}
-\int_{|x|>\delta} u[f_t+a_m^2\Delta f]dtdx=\int_{|x|>\delta}[u_t-a_m^2\Delta u]fdtdx+ \nonumber\\ a_m^2\int_{|x|=\delta} (u\nabla f-f\nabla u)\cdot\nu dt d\sigma(x)=a_m^2\int_{|x|=\delta}(u\nabla f-f\nabla u)\cdot\nu dt d\sigma(x),
\end{align}
where $\nu=x/|x|$ is a unit normal vector on the sphere $|x|=\delta$, and $d\sigma(x)$ is the surface Lebesgue measure on this sphere. Since for $|x|=\delta$
\[
u=AG(\delta/(a_m\sqrt{t}))+\const, \quad \nabla u\cdot\nu=-A\delta^{1-n}(a_m\sqrt{t})^{n-2}e^{-\frac{\delta^2}{4a_m^2t}},
\]
then, taking into account (\ref{lim0}), we find that in the limit as $\delta\to 0$
\[
\int_{|x|=\delta}(u\nabla f-f\nabla u)\cdot\nu d\sigma(x)\to A\omega_n(a_m\sqrt{t})^{n-2}f(t,0).
\]
It now follows from (\ref{sr2}) in the limit as $\delta\to 0$ that
\[
-\int_\Pi u[f_t+a_m^2\Delta f]dtdx=A\omega_na_m^n\int_0^{+\infty} t^{\frac{n}{2}-1} f(t,0)dt=\langle A\omega_na_m^nt^{\frac{n}{2}-1}\delta(x),f\rangle
\]
for all test functions $f\in C_0^\infty(\Pi)$. Hence, relation (\ref{sr1}) holds in $\D'(\Pi)$.

\medskip
Let $\xi_i$ be parameters of remaining phase transitions $|x|=\xi_i\sqrt{t}$ corresponding to the temperatures $u_i$, $i=1,\ldots,m-1$. Then $\xi_m<\xi_{m-1}<\cdots<\xi_1<\xi_0\doteq+\infty$ and in accordance with (\ref{form}) for
$i=0,\ldots,m-1$
\begin{equation}\label{soli}
v(\xi)=u_i+\frac{(u_{i+1}-u_i)(G(\xi/a_i)-G(\xi_i/a_i))}{G(\xi_{i+1}/a_i)-G(\xi_i/a_i)}, \quad \xi_{i+1}<\xi<\xi_i,
\end{equation}
where the value $G(\xi_0)=G(+\infty)=0$. The function $v(\xi)$, defined by relations (\ref{solm}), (\ref{soli}), is a solution of our Stefan problem.
Notice that this function decreases from $+\infty$ at $\xi=0$ to $u_0$ at $\xi=+\infty$ and, in particular, the initial condition (\ref{m2}) is satisfied. The unknown parameters $\xi_i$ are determined by Stefan condition (\ref{mSt}) on the surfaces $|x|=\xi_i\sqrt{t}$, which can be written as
\[
d_i\xi_i/2+k_iv'(\xi_i-0)-k_{i-1}v'(\xi_i+0)=0, \ i=1,\ldots,m.
\]
Computing the derivatives $v'(\xi_i\pm 0)$ from expressions (\ref{soli}), (\ref{solm}), we arrive at the system
\begin{align}\label{sys1}
d_i\xi_i/2+k_i\frac{(u_{i+1}-u_i)G'(\xi_i/a_i)}{a_i(G(\xi_{i+1}/a_i)-G(\xi_i/a_i))}- \nonumber\\
k_{i-1}\frac{(u_i-u_{i-1})G'(\xi_i/a_{i-1})}{a_{i-1}(G(\xi_i/a_{i-1})-G(\xi_{i-1}/a_{i-1}))}=0, \quad i=1,\ldots,m-1, \\
\label{sysm}
d_m\xi_m/2+k_m\frac{A}{a_m}G'(\xi_m/a_m)-\nonumber\\ k_{m-1}\frac{(u_m-u_{m-1})G'(\xi_m/a_{m-1})}{a_{m-1}(G(\xi_m/a_{m-1})-G(\xi_{m-1}/a_{m-1}))}, \quad i=m.
\end{align}
Like in one-dimensional case this system turns out to be gradient one, it coincides with the equality $\nabla E=0$, where the function
\begin{align}\label{E}
E(\bar\xi)=-\sum_{i=0}^{m-1}k_i(u_{i+1}-u_i)\ln (G(\xi_{i+1}/a_i)-G(\xi_i/a_i))\nonumber\\
+k_mAG(\xi_m/a_m)+\frac{1}{4}\sum_{i=1}^m d_i\xi_i^2, \quad \bar\xi=(\xi_1,\ldots,\xi_m)\in\Omega,
\end{align}
where $\Omega$ is an open convex cone in $\R^m$ consisting of vectors with strictly decreasing positive coordinates.
We are going to demonstrate that, like in the one-dimensional case, the function $E(\bar\xi)$ is coercive (but is not convex anymore).
This implies the existence of a point of its global minimum. Coordinates of this point determine a solution (\ref{solm}), (\ref{soli}) of our problem.

We first prove that the function $E(\bar\xi)$ is bounded from below.

\begin{lemma}\label{lem2}
There exists a constant $E_0$ such that $E(\bar\xi)\ge E_0$ for all $\bar\xi\in\Omega$.
\end{lemma}

\begin{proof}
Since the function $G(y)$ decreases,
\[G(\xi_{i+1}/a_i)-G(\xi_i/a_i)\le G(\xi_{i+1}/a_i)\le G(\xi_m/a_i), \quad i=0,\ldots,m-1.\]
Taking also into account that $k_i(u_{i+1}-u_i)>0$, $d_i\ge 0$, we obtain that $E(\bar\xi)\ge f(\xi_m)$, where the  function
\begin{equation}\label{f}
f(y)=-\sum_{i=0}^{m-1}k_i(u_{i+1}-u_i)\ln G(y/a_i)+k_mAG(y/a_m)
\end{equation}
is continuous on $(0,+\infty)$. As follows from (\ref{lim0}), the term $k_mAG(y/a_m)\to+\infty$ as $y\to 0+$ faster than
all the terms $\ln G(y/a_i)$. Therefore, $\displaystyle\lim_{y\to 0+} f(y)=+\infty$. It is also clear that $\displaystyle\lim_{y\to +\infty} f(y)=+\infty$. By these limit relations we see that $f(y)$ has a global minimum. Taking $E_0=\min\limits_{y>0} f(y)$, we complete the proof.
\end{proof}

Now we prove the coercivity of the function $E(\bar\xi)$.

\begin{theorem}\label{th4}
For each $c\in\R$ the set $\Omega_c=\{ \ \bar\xi\in\Omega \ | \ E(\bar\xi)\le c \ \}$ is compact.
\end{theorem}

\begin{proof}
If $\bar\xi=(\xi_1,\ldots,\xi_m)\in\Omega_c$ then $f(\xi_m)\le E(\bar\xi)\le c$, where $f(y)$ is the function (\ref{f}).
Since $f(y)\to +\infty$ as $y\to 0+$ there exists such $r_1>0$ that $f(y)>c$ for all $y\in (0,r_1)$. Then $\xi_m$ cannot be less than $r_1$:
\begin{equation}\label{b1}
\xi_m\ge r_1.
\end{equation}
In particular, this implies that
\[
k_i(u_{i+1}-u_i)\ln (G(\xi_{i+1}/a_i)-G(\xi_i/a_i))\le M=\max_{i=0,\ldots,m-1} k_i(u_{i+1}-u_i)\ln G(r_1/a_i).
\]
It follows from this bound that for each $i=0,\ldots,m-1$
\begin{align}\label{6}
-k_i(u_{i+1}-u_i)\ln (G(\xi_{i+1}/a_i)-G(\xi_i/a_i))\le\nonumber \\ E(\bar\xi)+(m-1)M\le c_1\doteq c+(m-1)M.
\end{align}
Taking $i=0$, we find
$
-k_0(u_1-u_0)\ln G(\xi_1/a_0)\le c_1.
$
This implies that $G(\xi_1/a_0)\ge e^{-c_1/(k_0(u_1-u_0))}>0$ and
\begin{equation}\label{b2}
\xi_1\le r_2\doteq a_0G^{-1}(e^{-c_1/(k_0(u_1-u_0))}).
\end{equation}
From relations (\ref{6}) with $i=1,\ldots,m-1$ it follows that
\begin{equation}\label{8}
G(\xi_{i+1}/a_i)-G(\xi_i/a_i)\ge q\doteq \min_{i=1,\ldots,m-1} e^{-c_1/(k_i(u_{i+1}-u_i))}>0
\end{equation}
Since $|G'(y)|\le y^{1-n}$, the function $G(y)$ is Lipschitz with constant $L=\alpha^{1-n}$ on the interval $[\alpha,+\infty)$. Choosing such $\alpha>0$ that $\alpha\max\limits_{i=1,\ldots,m-1} a_i\le r_1$, so that
$\xi_i/a_i>\xi_{i+1}/a_i\ge r_1/a_i\ge \alpha$, we derive from (\ref{8}) the inequalities
\[
L(\xi_i-\xi_{i+1})/a_i\ge G(\xi_{i+1}/a_i)-G(\xi_i/a_i)\ge q,
\]
which implies the bounds
\begin{equation}\label{b3}
\xi_i-\xi_{i+1}\ge\delta=\frac{q}{L}\min_{i=1,\ldots,m-1} a_i>0.
\end{equation}
In view of (\ref{b1}), (\ref{b2}), (\ref{b3}) we find that $\bar\xi\in K$, where
\[
K=\{\ \bar\xi=(\xi_1,\ldots,\xi_m)\in\R^m \ | \ r_2\ge\xi_1\ge\cdots\ge\xi_m\ge r_1>0, \ \xi_i-\xi_{i+1}\ge\delta, \ i=1,\ldots, m-1 \ \}
\]
is a compact subset of $\Omega$. By the continuity of $E(\bar\xi)$, we conclude that the set $\Omega_c$ is a closed subset of the compact $K$ and, therefore, is compact.
\end{proof}

As we have already demonstrated in the one-dimensional case, the coercivity of $E(\bar\xi)$ implies that there exists a global minimum $E(\bar\xi_0)=\min_\Omega E(\bar\xi)$. Coordinates of the point $\bar\xi_0$ provide the unknown parameters $\xi_i$,
$i=1,\ldots,m$ of the free boundaries in solution (\ref{solm}), (\ref{soli}). We establish the existence of solution.

\begin{theorem}\label{th5}
There exists a self-similar solution (\ref{solm}), (\ref{soli}) of the Stefan problem (\ref{m1}), (\ref{mSt}), (\ref{m2}), (\ref{gr}).
\end{theorem}

Observe that in the general case the function $E(\bar\xi)$ is not convex.
In fact, assume that $m=1$, so that
\[
E(\bar\xi)=E(\xi_1)=-k_0(u_1-u_0)\ln G(\xi_1/a_0)+Ak_1G(\xi_1/a_1)+d_1\xi_1^2/4.
\]
Since the function
\[
f(\xi_1)=-k_0(u_1-u_0)\ln G(\xi_1/a_0)+d_1\xi_1^2/4\to -\infty \ \mbox{ as } \xi_1\to 0+,
\]
it cannot be convex on $(0,+\infty)$. This implies that $E(\xi_1)=f(\xi_1)+Ak_1G(\xi_1/a_1)$ is not convex either for all sufficiently small $Ak_1>0$.

Nevertheless, a solution (\ref{solm}), (\ref{soli}) of Stefan problem (\ref{m1}), (\ref{mSt}), (\ref{m2}), (\ref{gr}) is unique. This can be proved like in \cite[Chapter V, \S~9]{LSU}. For the sake of completeness, we provide the details. First we notice that $u=v(|x|/\sqrt{t})$ is a solution of (\ref{m1}), (\ref{mSt}), (\ref{m2}), (\ref{gr}) if and only if it is a weak solution to the problem
\begin{equation}\label{par}
\beta(u)_t-\Delta_x \alpha(u)=\omega_nk_ma_m^{n-2}t^{\frac{n}{2}-1}\delta(x), \quad u(0,x)\equiv u_0,
\end{equation}
where $\alpha(u)$, $\beta(u)$ are strictly increasing functions on $(u_0,+\infty)$ linear on each interval $(u_i,u_{i+1})$,
$i=0,\ldots,m$, with slopes $\alpha'(u)=k_i$, $\beta'(u)=k_i/a_i^2$, and such that
\[
\alpha(u_i+)-\alpha(u_i-)=0, \quad \beta(u_i+)-\beta(u_i-)=d_i, \ i=1,\ldots,m.
\]
Remark that for all $u,v>u_0$, $u\not=v$
\begin{equation}\label{m3}
0<\frac{\alpha(u)-\alpha(v)}{\beta(u)-\beta(v)}\le \max a_i^2.
\end{equation}
Assume that $u_i=v_i(|x|/\sqrt{t})$, $i=1,2$, are two solutions. Then it follows from (\ref{par}) that
\begin{equation}\label{m4}
(\beta(u_1)-\beta(u_2))_t-\Delta_x(\alpha(u_1)-\alpha(u_2))=0 \ \mbox{ in } \D'(\Pi)
\end{equation}
As follows from expressions (\ref{solm}), (\ref{soli}), $P(t,x)\doteq \beta(v_1(\xi))-\beta(v_2(\xi))=\const $ for small $\xi=|x|/\sqrt{t}>0$ and $|P(t,x)|\le\const\cdot G(\xi/a_0)$ for large $\xi$. By the L'H\^{o}pital's rule
\begin{align*}
\lim_{y\to+\infty}-\frac{2y^{-1}G'(y)}{G(y)}=2\lim_{y\to+\infty}\frac{y^{-2}G'(y)-y^{-1}G''(y)}{G'(y)}=\\
2\lim_{y\to+\infty} [y^{-2}+y^{-1}((n-1)y^{-1}+y/2)]=1,
\end{align*}
and, therefore, $G(y)\sim -2y^{-1}G'(y)=2y^{-n}e^{-y^2/4}$ as $y\to+\infty$. This implies that for large $\xi=|x|/\sqrt{t}$ \[|P(t,x)|\le\const\cdot\xi^{-n}e^{-\xi^2/(4a_0^2)}.\]
We conclude that the function $P$ is bounded and $P(t,\cdot)\in L^2(\R^n)$, $\|P(t,\cdot)\|_2\le\const\cdot t^{n/4}$.
The similar statements hold for the function $Q=\alpha(u_1)-\alpha(u_2)$ because $Q=CP$, where $C=C(t,x)=(\alpha(u_1)-\alpha(u_2))/(\beta(u_1)-\beta(u_2))$ (if $u_1=u_2$ we set $C=0$) is a nonnegative bounded function, in view of (\ref{m3}).

Hence, we can apply (\ref{m4}) to a test function $f=f(t,x)$ from the Sobolev space $W_2^{1,2}(\Pi_T)$, $\Pi_T=(0,T)\times\R^n$ (so that $f,f_t,\nabla_x f,D_x^2 f\in L^2(\Pi_T)$) such that $f(T,x)=0$. As a result, we obtain the relation
\begin{equation}\label{m5}
\int_{\Pi_T} P(t,x)[f_t+C\Delta_x f]dtdx=0.
\end{equation}
Let $F(t,x)\in C_0^1(\Pi_T)$, $\varepsilon>0$, and $f^\varepsilon=f^\varepsilon(t,x)\in W_2^{1,2}(\Pi_T)$ be a solution of
the backward Cauchy problem
\begin{equation}\label{m6}
f_t+(C+\varepsilon)\Delta_x f=F, \quad f(T,x)=0.
\end{equation}
As is demonstrated in \cite{LSU}, such a solution exists and satisfies the estimate
\begin{equation}\label{m7}
|\Delta_x f^\varepsilon|_2\le C_0/\sqrt{\varepsilon},
\end{equation}
where $C_0$ is a constant independent of $\varepsilon$. It follows from (\ref{m5}) with $f=f^\varepsilon$ that
\begin{equation}\label{m8}
\int_{\Pi_T} P(t,x)F(t,x)dtdx=\varepsilon\int_{\Pi_T} P(t,x)\Delta_xf^\varepsilon(t,x)dtdx.
\end{equation}
By the Cauchy–Bunyakovsky inequality and (\ref{m7})
\[\left|\int_{\Pi_T} P(t,x)\Delta_xf^\varepsilon(t,x)dtdx\right|\le\|P\|_2\|\Delta_xf^\varepsilon\|_2\le C_0\|P\|_2/\sqrt{\varepsilon},\]
and the right-hand side of (\ref{m8}) vanishes as $\varepsilon\to 0$. We conclude that
\[\int_{\Pi_T} P(t,x)F(t,x)dtdx=0\]
for all $F(t,x)\in C_0^1(\Pi_T)$ and all $T>0$. This means that $P=0$ a.e. on $\Pi$, that is, $u_1=u_2$.

Thus, we have established the uniqueness.

\begin{theorem}\label{th6}
A self-similar solution (\ref{solm}), (\ref{soli}) of the Stefan problem (\ref{m1}), (\ref{mSt}), (\ref{m2}), (\ref{gr}) is unique. In particular, the function $E(\bar\xi)$ has only one critical point, the point of its global minimum.
\end{theorem}

\begin{remark}\label{rem3}
In the case when $u_0$ coincides with temperature of phase transition and the corresponding latent specific heat $d_0>0$ there exists another self-similar solution $u_1=v_1(\xi)$ of problem (\ref{m1}), (\ref{mSt}), (\ref{m2}), (\ref{gr}), which equals the constant $u_0$ in the domain $\xi_0<\xi<+\infty$ and has the form (\ref{solm}), (\ref{soli}) for $\xi<\xi_0$ (where $\xi_0$ is now finite). This solution contains the extra phase transition surface $\xi=\xi_0$. It can be easily verified that this new solution corresponds to the point of global minimum of the function
\begin{align}\label{E1}
E_1(\bar\xi)=-\sum_{i=0}^{m-1}k_i(u_{i+1}-u_i)\ln (G(\xi_{i+1}/a_i)-G(\xi_i/a_i))\nonumber\\
+k_mAG(\xi_m/a_m)+\frac{1}{4}\sum_{i=0}^m d_i\xi_i^2, \quad \bar\xi=(\xi_0,\ldots,\xi_m)\in\Omega,
\end{align}
where the cone $\Omega\subset\R^{m+1}$ consists of points with strictly decreasing positive coordinates. We notice that the only difference between formulas (\ref{E1}) and (\ref{E}) is the extra term $d_0\xi_0^2/4$. This term guarantees coercivity of $E_1$ and the existence of solution $u_1$. Moreover, the uniqueness of this solution (of the prescribed above form) can be established as in the proof of Theorem~\ref{th6}.

Observe that both the functions $u,u_1$ are solutions of (\ref{par}), where the increasing functions $\alpha(u),\beta(u)$ are now defined up to $u_0$ and satisfy the conditions $\alpha(u_0)=\alpha(u_0+)$, $\beta(u_0)=\beta(u_0+)-d_0$. This seems surprising because of the uniqueness statement of Theorem~\ref{th6}. But there is no contradiction here. In fact, as is easy to verify,
\[P=\beta(v(\xi))-\beta(v_1(\xi))\mathop{\to}_{\xi\to+\infty} d_0>0,
\]
therefore, $P(t,\cdot)\notin L^2(\R^n)$, and the reasoning used in the proof of Theorem~\ref{th6} is not applicable.
\end{remark}

\end{document}